\documentclass[english]{article}
\usepackage[latin9]{inputenc}
\usepackage{pifont}
\usepackage{amsmath}
\usepackage{amssymb}
\usepackage{esint}

\makeatletter

\title{Newtheorem and theoremstyle test}
\author{Michael Downes\\updated by Barbara Beeton}

\usepackage[exercise]{amsthm}

\newtheorem{thm}{Theorem}[section]

\newtheorem{lem}[thm]{Lemma}

\theoremstyle{remark}
\newtheorem*{rmk}{Remark}

\theoremstyle{plain}
\newtheorem{Def}{Definition}

\newtheoremstyle{note}
  {3pt}
  {3pt}
  {}
  {}
  {\itshape}
  {:}
  {.5em}
  {}

\theoremstyle{note}

\newtheoremstyle{citing}
  {3pt}
  {3pt}
  {\itshape}
  {}
  {\bfseries}
  {.}
  {.5em}
  {\thmnote{#3}}

\theoremstyle{citing}

\newtheoremstyle{break}
  {9pt}
  {9pt}
  {\itshape}
  {}
  {\bfseries}
  {.}
  {\newline}
  {}

\theoremstyle{break}

\theoremstyle{exercise}

\swapnumbers
\theoremstyle{plain}

\let\lvert=|\let\rvert=|

\usepackage{babel}

\usepackage{babel}

\usepackage{babel}

\usepackage{babel}

\makeatother

\usepackage{babel}
\begin{document}

\title{$N$ -Laplacian and $N/2$-Hessian type equations with exponential
reaction term and measure data}

\author{Shiguang Ma and Zijian Wang}
\maketitle
\begin{abstract}
In this article, we will prove existence results for the equations
of the type $-\Delta_{N}u=H_{l}(u)+\mu$ and $F_{\frac{N}{2}}[-u]=H_{l}(u)+\mu$
in a bounded domain $\Omega$, with Dirichlet boundary condition,
where the source term $H_{l}(r)$ takes the form $e^{r}-\sum_{j=0}^{l-1}\frac{r^{j}}{j!}$
and $\mu$ is a nonnegative Radon measure.
\end{abstract}

\section{Introduction}

The following two types of equations appear in areas as quasi-regular
mappings, non-Newtonian fluids and reaction-diffusion problems, etc, 

\begin{align}
-\Delta_{p}u & =F(u,x),\label{p-Laplacian-F}\\
F_{k}[-u] & =F(u,x),\label{k-hessian-F}
\end{align}
where $\Delta_{p}u={\rm div}(|\nabla u|^{p-2}\nabla u)$ is the $p$-Laplacian
($p>1$), and $F_{k}[-u]$ is the $k$-Hessian $(k=1,2,\cdots,n)$
defined by 
\[
F_{k}[-u]=\sum_{1\le i_{1}<\cdots<i_{k}\le n}\lambda_{i_{1}}\cdots\lambda_{i_{k}},
\]
where $\lambda_{1},\cdots,\lambda_{n}$ are the eigenvalues of the
Hessian matrix $-D^{2}u$. 

For $p$-Laplacian equation, we may refer to \cite{Hein+Kilp+Mart,Iwaniec-Martin,Kilp and Maly Acta,Mazya1,Maly and Ziemer,Serrin1,Serrin2,Serrin-Zou acta,Trudinger-Wang4}
for the existence and regularity theory, estimates for supersolutions
and Wiener criterion, ect.

For $k$-Hessian equation, one may refer to \cite{Caffarelli-Nirenberg-Spruck,Gilberg-Trudinger,Guan1,Ivochkina,Krylov,Trudinger,Trudinger-Wang1,Trudinger-Wang2,Trudinger-Wang3,Trudinger-Wang4,Urbas,Labutin}
for related knowledge.

Here we are interested in quasilinear and fully nonlinear equations
(\ref{p-Laplacian-F}) and (\ref{k-hessian-F}) and the corresponding
inequalities
\begin{align*}
-\Delta_{p}u & \ge F(u,x),\,\,{\rm and}\,\,F_{k}[-u]\ge F(u,x),\,\,u\ge0\,\,\in\Omega.
\end{align*}
The inequalities can also be written as inhomogeneous equations with
measure data,
\begin{equation}
-\Delta_{p}u=F(u,x)+\mu,\,\,F_{k}[-u]=F(u,x)+\mu,\,\,u\ge0\,\,{\rm in}\,\,\Omega,\label{equation with measure data}
\end{equation}
where $\mu$ is a nonnegative Borel measure on $\Omega$. 

When $F(u,x)=u^{q}$, a major breakthrough on the existence of solutions
to (\ref{equation with measure data}) is made by \cite{Phuc-Verbitsky1}.
The difficulties in studies of such equations and inequalities lie
in the competing nonlinearities. The argument of \cite{Phuc-Verbitsky1}
relies largely on nonlinear potential theory \cite{Kilp and Maly Acta,Labutin}.
Several necessary and sufficient conditions are given in \cite{Phuc-Verbitsky1}.
For example, let's revise one of the theorems in \cite{Phuc-Verbitsky1}.
First, for $s>1,0\le\alpha<\frac{N}{s}$ and $0<T\le\infty$, the
$T$-truncated Wolff potential of a nonnegative Radon measure $\mu$
is defined by 
\[
W_{\alpha,s}^{T}[\mu](x)=\int_{0}^{T}(\frac{\mu(B_{t}(x))}{t^{N-\alpha s}})^{\frac{1}{s-1}}\frac{dt}{t}.
\]

\begin{thm}\label{Thm 2.10 of PV1} (Theorem 2.10 of \cite{Phuc-Verbitsky1})

The following two items are equivalent:
\begin{enumerate}
\item there exists a nonnegative renormalized solution $u\in L^{q}(\Omega)$
\[
\begin{cases}
-\Delta_{p}u=u^{q}+\varepsilon\omega, & {\rm in}\,\,\Omega,\\
u=0 & {\rm on}\,\,\partial\Omega
\end{cases}
\]
for some $\varepsilon>0$;
\item For all compact sets $E\subset\Omega,$
\begin{equation}
\omega(E)\le CCap_{G_{p},\frac{q}{q-p+1}}(E);\label{second equivalent conditon}
\end{equation}
\item The testing inequality 
\[
\int_{B}[W_{1,p}^{2R}\omega_{B}(x)]^{q}\le C\omega(B)
\]
holds for all balls $B$ such that $B\cap{\rm supp}\omega\neq\emptyset;$
\item There exists a constant $C$ such that 
\[
W_{1,p}^{2R}(W_{1,p}^{2R}\omega)^{q}(x)\le CW_{1,p}^{2R}\omega(x),a.e.on\,\Omega.
\]
\end{enumerate}
\end{thm}

The equivalence conditions may be understood in the following way:
the measure can only concentrate in a relatively mild way (Condition
2) such that $u^{q}$ (or equivalently $(W_{1,p}^{2R}\omega)^{q}$)
plays a minor role compared with $\varepsilon\omega$(Condition 3
and $4$). 

For generalizations, it is natural to consider the case when $F(u,x)$
is of exponential type. In \cite{Nguyen-Veron}, the authors considered
the case when $F(u,x)=H_{l}(\alpha u^{\beta}),\alpha>0,\beta\ge1$,
where 
\begin{equation}
H_{l}=e^{r}-\sum_{j=0}^{l-1}\frac{r^{j}}{j!}.\label{Hl}
\end{equation}
 Also we notice that, they consider $p$-Laplacian case $1<p<N$ and
$k$-Hessian case $1<k<\frac{N}{2}$. To get the existence of the
solutions, the condition of the type 
\begin{equation}
\|M_{p,2{\rm diam}(\Omega)}^{\frac{(p-1)(\beta-1)}{\beta}}[\mu]\|_{L^{\infty}(\mathbb{R}^{N})}\le M\label{NV type condition}
\end{equation}
 is imposed, where 
\[
M_{\alpha,T}^{\eta}[\mu](x)=\sup\{\frac{\mu(B_{t}(x))}{t^{N-\alpha}h_{\eta}(t)}:0<t\le T\},
\]
where 
\[
h_{\eta}(t)=(-\ln t)^{-\eta}\chi_{(0,2^{-1}]}(t)+(\ln2)^{-\eta}\chi_{[2^{-1},\infty)}(t).
\]

More general results are given in \cite{Nguyen1}. For example, the
parabolic equation is considered with the right hand side containing
gradient terms. Also (\ref{NV type condition}) type conditions are
required. 

However, up to our knowledge, the following type equations are not
studied
\[
-\Delta_{N}u=H_{l}(u)+\mu,\,\,F_{\frac{N}{2}}[-u]=H_{l}(u)+\mu,\,\,x\in\Omega,
\]
where $\Omega$ is a bounded domain and $\mu$ is a nonnegative Radon
measure and $H_{l}(u)$ is given by (\ref{Hl}). The operators $\Delta_{N}$
and $F_{\frac{N}{2}}$ are both borderline operators, which have fundamental
solutions of logarithm type. Our main results are list below.

\begin{thm}\label{Main Thm p-Laplace}

Let $\Omega\subset\mathbb{R}^{N}$ be a bounded domain, $l\in\mathbb{N}$
and $l\ge N$ and $\mu$ is a nonnegative Radon measure supported
in $\Omega$. Then there exists $M>0$ depending only on $N,l$ such
that if
\[
\mu(\Omega)\le M,
\]
then the following Dirichlet problem
\begin{equation}
\begin{cases}
-\Delta_{N}u & =H_{l}(u)+\mu,\,\,{\rm in}\,\,\Omega,\\
u & =0,\,\,{\rm on}\,\,\partial\Omega,
\end{cases}\label{N-Laplace u equation}
\end{equation}
admits a nonnegative renormalized solution $u$ which satisfies
\[
u(x)\le C(N,p)W_{1,N}^{2{\rm diam}(\Omega)}[\bar{\mu}](x),\forall x\in\Omega.
\]

\end{thm}

Concerning the $k$-Hessian operator we recall some notions introduced
by Trudinger and Wang (\cite{Trudinger-Wang1,Trudinger-Wang2,Trudinger-Wang3}).
For $k=1,\cdots,N$ and $u\in C^{2}(\Omega)$ the $k$-Hessian operator
$F_{k}$ is defined by 
\[
F_{k}[u]=S_{k}(\lambda(D^{2}u)),
\]
where $\lambda(D^{2}u)=\lambda=(\lambda_{1},\cdots,\lambda_{N})$
denotes the eigenvalues of $D^{2}u$, and 
\[
S_{k}=\sum_{1\le i_{1}<\cdots<i_{k}\le N}\lambda_{i_{1}}\cdots\lambda_{i_{k}}
\]
 is the $k$th elementary symmetric polynomial of $\lambda$. For
equations of type (\ref{k-hessian-F}), we always seek for $k$-admissible
solutions, which satisfy
\[
\lambda(-\nabla^{2}u)\in\Gamma_{k}=\{\lambda;S_{1}(\lambda)>0,\cdots,S_{k}(\lambda)>0\}.
\]
For the existence of such solutions, we assume $\partial\Omega$ satisfies
a uniformly $(k-1)$-convex condition, that is 
\[
S_{j}(\kappa)>0,\,\,{\rm on}\,\,\partial\Omega,\,\,{\rm for}\,\,j=1,\cdots,k-1.
\]
where $\kappa=(\kappa_{1},\cdots,\kappa_{N-1})$ denote the principal
curvatures of $\partial\Omega$ with respect to its outward normal. 

We refer to Definition \ref{k-subharmonic} for the definition of
$k$-subharmonic functions and $\Phi^{k}(\Omega)$. 

\begin{thm}\label{Main theorem k hessian}

Let $l\in\mathbb{N}$ and $l>\frac{N}{2}$. Let $\Omega$ be a bounded
uniformly $(\frac{N}{2}-1)$ convex in $\mathbb{R}^{N}$. Let $\mu$
be a nonnegative Radon measure, which has compact support in $\Omega$.
Then there exists $M>0$, such that if 
\[
\mu(\Omega)\le M
\]
then the following Dirichlet problem 
\begin{equation}
\begin{cases}
F_{k}[u] & =H_{l}(u)+\mu,\,\,{\rm in}\,\,\Omega,\\
u & =0,\,\,{\rm on}\,\,\partial\Omega,
\end{cases}\label{N/2-Hessian equation}
\end{equation}
admits a nonnegative solution $u\in\Phi^{k}(\Omega)$, continuous
near $\partial\Omega$, which satisfies
\[
u(x)\le2K_{2}W_{\frac{2N}{N+2},\frac{N+2}{2}}^{2{\rm diam}(\Omega)}[\mu](x)+b,\forall x\in\Omega.
\]

\end{thm}

\begin{rmk}Here we note that a major difference between our results
Theorem \ref{Main Thm p-Laplace}, \ref{Main theorem k hessian} and
the results of Theorem \ref{Thm 2.10 of PV1}, the results of \cite{Nguyen-Veron}
is that, we do not need the assumptions of the type (\ref{second equivalent conditon})
or (\ref{NV type condition}). We only need $\mu(\Omega)$ to be small.
\end{rmk}

\section{Estimates of the potentials}

We use $B_{r}(a)$ to denote the ball of radius $r$ and with center
$a$. We use $\mu_{L}$ to denote the Lebesgue measure of $\mathbb{R}^{N}$.
We use ${\rm diam}(\Omega)$ to denote the diameter of , i.e. $\sup\{d(x,y);x,y\in\Omega\}$.
For a finite Radon measure supported on a domain $\Omega$, we may
sometimes consider it as a Radon measure on $\mathbb{R}^{n}$ by extending
it to be $0$ outside $\Omega$. If $D$ is a subset of $\Omega$,
we will use $\mu_{D}$ to represent $\mu$ restricted on $D$. 

We will usually use $C$ to denote a uniform positive constant and
use $C(R_{1},\cdots,R_{s})$ to denote a positive constant which depends
on $R_{1}$ to $R_{s}$. 

\begin{Def}

Given a nonnegative Radon measure $\mu$ on $\mathbb{R}^{N}$, we
define the maximal function of $\mu$ as 
\[
M_{\mu}(x):=\sup_{r>0}\frac{\mu(B_{r}(x))}{|B_{r}(x)|},\,\,x\in E.
\]

\end{Def}

The following weak $(1,1)$ type Hardy-Littlewood maximal inequality
about measure is proved in \cite{Folland}. 

\begin{lem}\label{Maximal function Weak 1 1}

Suppose $\mu$ is a finite Radon measure on $\mathbb{R}^{N}$. There
exists $C(N)>0$ such that 
\[
|\{x\in\mathbb{R}^{N}:M_{\mu}(x)>\lambda\}|\le\frac{C(N)}{\lambda}\mu(\mathbb{R}^{N}),\,\,\forall\lambda>0.
\]

\end{lem}

To prove our main theorem, we now generalize a result of Ma and Qing,
\cite[Proposition 4.1]{Ma-Qing-CVPDE}. 

\begin{lem}\label{Brezis-Merle type inequality}

Let $\Omega\subset\mathbb{R}^{N}$ be a bounded domain with diameter
$D$. And let $\mu$ be a finite nonnegative Radon meaure support
in $\Omega$. We let $1<p\le N,\alpha=\frac{N}{p}$. Then, for any
$\delta\in(0,1)$ 
\[
\int_{\Omega}{\rm exp}\left(\frac{N(1-\delta)W_{\alpha,p}^{R}[\mu](x,D)}{\mu(\Omega){}^{\frac{1}{p-1}}}\right)dx\le\frac{c(N)|B(0,D)|}{\delta^{N+1}}.
\]

\end{lem}

\begin{proof}

The proof follows similarly from \cite[Proposition 4.1]{Ma-Qing-CVPDE}.
For the convenience of the readers, we present a proof here. To start,
we choose $p-1<q<N$ and assume $\mu(B(x,D))=1.$ Then 
\begin{align*}
W_{\alpha,p}^{D}[\mu](x) & \le\int_{0}^{D}\mu(B(x,t))^{\frac{1}{q}}\frac{dt}{t}\\
 & =\mu(B(x,t))^{\frac{1}{q}}\log t|_{0}^{D}+\int_{0}^{D}\log\frac{1}{t}d\mu(B(x,t))^{\frac{1}{q}}.
\end{align*}
Since 
\[
\mu(B(x,t))\le|B(0,t)|M_{\mu}(x)=\omega_{N}t^{N}M_{\mu}(x)
\]
we have 
\[
\mu(B(x,t))^{\frac{1}{q}}\log t|_{0}^{D}=\log D
\]
almost everywhere.

Then by Jensen's inequality 
\begin{align*}
\exp(W_{\alpha,p}^{D}[\mu](x)) & \le D\int_{0}^{D}\frac{1}{t}d\mu(B(x,t))^{\frac{1}{q}}\\
 & \le D(\frac{1}{t}\mu(B(x,t))^{\frac{1}{q}}|_{0}^{D}+\int_{0}^{D}\mu(B(x,t))^{\frac{1}{q}}\frac{1}{t^{2}}dt).
\end{align*}
Since $q<N$, we know $\frac{1}{t}\mu(B(x,t))^{\frac{1}{q}}|_{0}^{D}=\frac{1}{D}$
almost everywhere.

Then 
\begin{align*}
\exp(W_{\alpha,p}^{D}[\mu](x)) & \le1+D^{\frac{N}{q}}\frac{1}{\frac{N}{q}-1}\omega_{N}^{\frac{1}{q}}M_{\mu}(x)^{\frac{1}{q}}.
\end{align*}
So for $\lambda\ge2$, we have 
\begin{align*}
|\{x\in\Omega;\exp(W_{\alpha,p}^{D}[\mu](x))\ge\lambda\}| & \le|\left\{ x\in\Omega;M_{\mu}(x)\ge\frac{(N-q)^{q}\lambda^{q}}{2^{q}q^{q}|B(0,D)|}\right\} |\\
 & \le\frac{C(N)2^{q}q^{q}|B(0,D)|}{(N-q)^{q}\lambda^{q}},
\end{align*}
where the last inequality is due to Lemma \ref{Maximal function Weak 1 1}. 

For $\delta\in(0,1)$, we fix some $q>\max\{p-1,N(1-\frac{\delta}{2})\}$
and $q<N$
\begin{align*}
\int_{\Omega}{\rm exp}(N(1-\delta)W_{\alpha,p}^{D}[\mu](x))dx & =\int_{0}^{+\infty}|\{x\in\Omega;{\rm exp}(W_{\alpha,p}^{\mu}(x,D))\ge t^{\frac{1}{N(1-\delta)}}\}|\\
 & \le\int_{2^{N(1-\delta)}}^{+\infty}\frac{C(N)2^{q}q^{q}|B(0,D)|}{(N-q)^{q}t^{\frac{q}{N(1-\delta)}}}dt+\int_{0}^{2^{N(1-\delta)}}|\Omega|dt\\
 & \le C(N)\frac{|B(0,D)|}{\delta^{N+1}}.
\end{align*}
For the case $\mu(\Omega)\ne1$, we consider $\frac{\mu}{\mu(\Omega)}$
instead and finally prove the claim.

\end{proof}

\begin{lem}\label{Iteration potential estimate}

Let $2\le p\le N$ and $\alpha=\frac{N}{p}$. Let $\mu$ be a nonnegative
Radon measure supported on $\Omega$ satisfies 
\[
\mu(\Omega)\le1
\]
 and set $\bar{\mu}=\frac{\mu_{L}}{\mu_{L}(B_{10{\rm diam}(\Omega)})}+\mu$.
We fix $R\in(\frac{1}{10}{\rm diam}(\Omega),10{\rm diam}(\Omega))$.
Then there exist uniform positive constants $C_{1},\delta_{0}$ such
that 
\begin{equation}
\|W_{\alpha,p}^{R}[\exp(\delta_{0}W_{\alpha,p}^{R}[\bar{\mu}])]\|_{L^{\infty}(\mathbb{R}^{N})}\le C_{1},\label{potential L infinity estimate}
\end{equation}
and 
\begin{equation}
W_{\alpha,p}^{R}[\exp(\delta_{0}W_{\alpha,p}^{R}[\bar{\mu}])]\le C_{1}W_{\alpha,p}^{R}[\bar{\mu}]\label{Potential bounds}
\end{equation}
in $\Omega$. 

\end{lem}

\begin{proof}

First we have 
\[
\bar{\mu}(\Omega)\le2.
\]
For $0<r<R$, and $y\in\Omega$, we have 
\begin{align*}
W_{\alpha,p}^{R}[\bar{\mu}](y) & =W_{\alpha,p}^{r}[\bar{\mu}](y)+\int_{r}^{R}\bar{\mu}(B_{t}(y))^{\frac{1}{p-1}}\frac{dt}{t}\\
 & \le W_{\alpha,p}^{r}[\bar{\mu}](y)+C(\log R-\log r)\\
 & \le W_{\alpha,p}^{r}[\bar{\mu}](y)+C\log\frac{1}{r}+C(R).
\end{align*}
From $\exp((a+b)/2)\le(\exp(a)+\exp(b))/2$, for some $\theta\in(0,2^{-\frac{1}{N-1}}]$
to be fixed later, we have 
\begin{align*}
\exp(\theta N(1-\delta)2^{-2}W_{\alpha,p}^{R}[\bar{\mu}(y)])\le & \frac{1}{2}\exp(\frac{1}{2}N(1-\delta)W_{\alpha,p}^{r}[\bar{\mu}](y))\\
 & +C(R)\exp(\theta C\log\frac{1}{r})\\
\le & \frac{1}{2}\exp(\frac{1}{2}N(1-\delta)W_{\alpha,p}^{r}[\bar{\mu}](y))\\
 & +C(R)r^{-\theta C_{2}}.
\end{align*}
Then for $r>0$, from Lemma \ref{Brezis-Merle type inequality}, we
have 
\begin{align*}
\int_{B_{r}(x)}\exp(\frac{1}{2}N(1-\delta)W_{\alpha,p}^{r}[\bar{\mu}](y))dy\le & \int_{B_{r}(x)}\exp(\frac{N(1-\delta)W_{\alpha,p}^{r}[\bar{\mu}](y)}{\bar{\mu}(B_{2r}(x))^{\frac{1}{p-1}}})dy\\
\le & Cr^{N}.
\end{align*}
 Therefore by choosing $\theta>0$ small such that $\theta C_{2}\le\frac{1}{2}$
we have 
\begin{align*}
W_{\alpha,p}^{R}[\exp(\theta N(1-\delta)2^{-2}W_{\alpha,p}^{R}[\bar{\mu}][y])](x) & \le\int_{0}^{R}(Cr^{N}+Cr^{-\theta C_{1}}r^{N})^{\frac{1}{p-1}}\frac{dr}{r}\\
 & \le C.
\end{align*}

Then we may choose $\delta_{0}=\frac{N(1-\delta)}{8C_{2}}$ and get
(\ref{potential L infinity estimate}). 

Also note that 
\begin{align*}
W_{\alpha,p}^{R}[\bar{\mu}](x) & =\int_{0}^{R}\bar{\mu}(B(x,t))^{\frac{1}{p-1}}\frac{dt}{t}\\
 & \ge\int_{0}^{R}(\frac{\mu_{L}(B(x,t))}{\mu_{L}(B_{10{\rm diam}(\Omega)})})^{\frac{1}{p-1}}\frac{dt}{t}\\
 & \ge C>0.
\end{align*}
Then (\ref{Potential bounds}) follows.

\end{proof}

Then we have 

\begin{thm}\label{Theorem u_m upper bound}

Let $2\le p\le N$ and $\alpha=\frac{N}{p}$. Suppose $\mu$ is a
finite nonnegative Radon measure on $\Omega$. Assume $R\in[\frac{1}{10}{\rm diam}(\Omega),10{\rm diam}(\Omega)]$
and $K>0$ are positive real numbers. We fix an integer $l>p-1$.
Suppose $\{u_{m}\}$ is a sequence of nonnegative measurable functions
on $\mathbb{R}^{N}$ which satisfies
\begin{align}
u_{0} & \le KW_{\alpha,p}^{R}[\mu],\label{u0 bound}\\
u_{m+1} & \le KW_{\alpha,p}^{R}[H_{l}(u_{m})+\mu],\forall m\in\mathbb{N}.\label{u_=00007Bm+1=00007D bounded by u_m}
\end{align}
Then there exists $M>0$ depending only on $N,l,K,\delta_{0},C_{1}$
such that if 
\begin{align}
 & \mu(\Omega)\le M,\label{mu upper bdd}\\
\bar{\mu}= & M\frac{\mu_{L}}{\mu_{L}(B_{10{\rm diam}(\Omega)}(0))}+\mu,\label{bar mu definition}
\end{align}
there holds 
\begin{equation}
H_{l}(2KW_{\alpha,p}^{R}[\bar{\mu}])\in L^{s}(\Omega),\,\,{\rm for}\,\,{\rm some}\,\,s>1\label{Hl(W) estimate}
\end{equation}
and 
\begin{equation}
u_{m}\le2KW_{\alpha,p}^{R}[\bar{\mu}],\forall m\in\mathbb{N}.\label{u_m bound}
\end{equation}

\end{thm}

\begin{proof}

First prove that, we may choose $M>0$ small such that when $\mu$
and $\bar{\mu}$ satisfy (\ref{mu upper bdd})(\ref{bar mu definition}),
then

\begin{align*}
H_{l}(2KW_{\alpha,p}^{R}[\bar{\mu}])\in L^{s}(\Omega), & \,\,{\rm for}\,\,{\rm some}\,\,s>1,\\
W_{\alpha,p}^{R}[H_{l}(2K\cdot W_{\alpha,p}^{R}[\bar{\mu}])] & \le W_{\alpha,p}^{R}[\bar{\mu}].
\end{align*}
By Lemma \ref{Iteration potential estimate}, there exists $C>0$
and $\delta_{0}>0$ independent of $\bar{\mu}$ such that $\exp(\delta_{0}W_{\alpha,p}^{R}[M^{-1}\bar{\mu}])$
is integrable in $\Omega$ and 
\[
W_{\alpha,p}^{R}[\exp(\delta_{0}W_{\alpha,p}^{R}(M^{-1}\bar{\mu}))]\le C_{1}W_{\alpha,p}^{R}[M^{-1}\bar{\mu}]\,\,{\rm in}\,\,\bar{\Omega}.
\]
Here in fact we choose $\delta_{0}$ even smaller, so that $\exp(\delta_{0}W_{\alpha,p}^{R}[M^{-1}\bar{\mu}])\in L^{s}(\Omega)$
for some $s>1$. 

Since $l>p-1\ge1$, then for all $t\ge0$ and $0<\theta\le1$
\begin{align*}
\theta^{-l}H_{l}(t) & =\theta^{-l}(e^{t}-\sum_{i=1}^{l-1}\frac{t^{j}}{j!})=\theta^{-l}\sum_{i=l}^{\infty}\frac{t^{j}}{j!}\le\sum_{i=l}^{\infty}\frac{(\theta^{-1}t)^{j}}{j!}=H_{l}(\theta^{-1}t).
\end{align*}
It follows that 
\begin{align*}
 & W_{\alpha,p}^{R}[M^{-\frac{1}{2}(\frac{l}{p-1}+1)}H_{l}(\delta_{0}M^{-\frac{1}{2}(\frac{1}{p-1}-\frac{1}{l})}W_{\alpha,p}^{R}[\bar{\mu}])]\\
\le & W_{\alpha,p}^{R}[H_{l}(\delta_{0}M^{-\frac{1}{p-1}}W_{\alpha,p}^{R}[\bar{\mu}])]\\
\le & W_{\alpha,p}^{R}[\exp(\delta_{0}W_{\alpha,p}^{R}[M^{-1}\bar{\mu}])]\\
\le & C_{1}M^{-\frac{1}{p-1}}W_{\alpha,p}^{R}[\bar{\mu}].
\end{align*}
 Then 
\[
W_{\alpha,p}^{R}[H_{l}(\delta_{0}M^{-\frac{1}{2}(\frac{1}{p-1}-\frac{1}{l})}W_{\alpha,p}^{R}[\bar{\mu}])]\le C_{1}M^{\frac{1}{2(p-1)}(\frac{l}{p-1}-1)}W_{\alpha,p}^{R}[\bar{\mu}].
\]
Then if we choose $M=\min\{(\frac{\delta_{0}}{2K})^{\frac{2(p-1)l}{l-p+1}},C_{1}^{-\frac{2(p-1)^{2}}{l-p+1}}\}$,
we have $H_{l}(2KW_{\alpha,p}^{R}[\bar{\mu}])\in L^{s}(\Omega)$ for
some $s>1$ and 
\begin{equation}
W_{\alpha,p}^{R}[H_{l}(2KW_{\alpha,p}^{R}[\bar{\mu}])]\le W_{\alpha,p}^{R}[\bar{\mu}].\label{potential iteration inequality}
\end{equation}
Now we will prove (\ref{u_m bound}) by induction. From (\ref{u0 bound}),
(\ref{u_m bound}) holds for $m=0$. Now we assume (\ref{u_m bound})
holds with $m=k$. Then from (\ref{u_=00007Bm+1=00007D bounded by u_m})
and $(\ref{potential iteration inequality})$
\begin{align*}
u_{k+1} & \le KW_{\alpha,p}^{R}[H_{l}(u_{k})+\bar{\mu}]\\
 & \le KW_{\alpha,p}^{R}[H_{l}(2KW_{\alpha,p}^{R}[\bar{\mu}])+\bar{\mu}]\\
 & \le KW_{\alpha,p}^{R}[H_{l}(2KW_{\alpha,p}^{R}[\bar{\mu}])]+KW_{\alpha,p}^{R}[\bar{\mu}]\\
 & \le2KW_{\alpha,p}^{R}[\bar{\mu}].
\end{align*}
Then we completes the proof. 

\end{proof}

\section{Quasilinear equations with Dirichlet boundary conditions}

In this section, we use the estimates in the last section to study
the existence of solutions to 
\begin{equation}
\begin{cases}
-\Delta_{N}u & =H(u)+\mu,\,\,{\rm in}\,\Omega\\
u & =0,\,\,{\rm on}\,\partial\Omega.
\end{cases}\label{N-Laplacian equation}
\end{equation}
First we need some basic knowledges related to $p$-Laplacian equations.

Let $\Omega$ be a bounded domain in $\mathbb{R}^{N}$, $1<p\le N$
and $p'=\frac{p}{p-1}$. We use $\mathfrak{M}_{b}(\Omega)$ to deonte
the set of all finite Radon measures on $\Omega$. If $\mu\in\mathfrak{M}_{b}(\Omega)$,
by Jordan decomposition, we denote $\mu^{+}$ and $\mu^{-}$ as the
positive part and negative part of $\mu$ and use $\mathfrak{M}_{b}^{+}$
to denote the set of nonnegative Radon measures . We use $\mathfrak{M}_{0,p}(\Omega)$
to denote the space of measures in $\Omega$ which are absolutely
continuous with repect to $c_{1,p}^{\Omega}$-capacity, which for
any compact subset $K\subset\Omega$, is defined by 
\[
c_{1,p}^{\Omega}(K)=\inf\left\{ \int_{\Omega}|\nabla\phi|^{p}dx;\phi\in C_{0}^{\infty}(\Omega),\phi\ge\chi_{K}\right\} ,
\]
where 
\[
\chi_{E}=\begin{cases}
1, & x\in E,\\
0, & {\rm otherwise}.
\end{cases}
\]
Also, we use $\mathfrak{M}_{s,p}(\Omega)$ to denote the space of
measures in $\Omega$ supported on a set of zero $c_{1,p}^{\Omega}$-capacity.
We know, any $\mu\in\mathfrak{M}_{b}(\Omega)$ can be decomposed as
$\mu=\mu_{0}+\mu_{s}$ in a unique way, such that $\mu_{0}\in\mathfrak{M}_{0,p}\cap\mathfrak{M}_{b}$
and $\mu_{s}\in\mathfrak{M}_{s,p}\cap\mathfrak{M}_{b}$. And it is
well known that $\mu_{0}$ part can be written as $f-{\rm div}g$
where $f\in L^{1}(\Omega)$ and $g\in L^{p'}(\Omega,\mathbb{R}^{N})$,
see \cite{Dal-Murat-Orsina-Prignet}.

For $k>0$ and $s\in\mathbb{R}$, we set $T_{k}(s)=\max\{\min\{s,k\},-k\}.$
From \cite{Ben-Boc-Gal-Gar-Pie-Vaz}, if $u$ is a measurable function
on $\Omega,$ finite almost everywhere such that $T_{k}(u)\in W_{{\rm loc}}^{1,p}(\Omega)$
for any $k>0$, there exists a measurable function $v(x):\Omega\to\mathbb{R}^{N}$
such that $\nabla T_{k}(u)=\chi_{\{|u|\le k\}}(x)v(x)$ almost everywhere
in $\Omega$ and for all $k>0$. We define $\nabla u=v$. Now we recall
the definition of a renormalized solution given in \cite{Dal-Murat-Orsina-Prignet}. 

\begin{Def}

Let $1<p\le N$ and $\mu=\mu_{0}+\mu_{s}\in\mathfrak{M}_{b}(\Omega)$
. A measurable function $u$ defined in $\Omega$ and finite almost
everywhere is called a renormalized solution of 
\begin{equation}
\begin{cases}
-\Delta_{p}u & =\mu,\,\,{\rm in}\,\,\Omega,\\
u & =0,\,\,{\rm on}\,\,\partial\Omega,
\end{cases}\label{p-Laplace equation}
\end{equation}
 if $T_{k}(u)\in W_{0}^{1,p}(\Omega)$ for any $k>0$, $|\nabla u|^{p-1}\in L^{r}(\Omega)$
for any $1<r<\frac{n}{n-1}$, and $u$ has the property that for any
$k>0$ there exist $\lambda_{k}^{+}$ and $\lambda_{k}^{-}$ belonging
to $\mathfrak{M}_{b}^{+}\cap\mathfrak{M}_{0}(\Omega)$, respectively
concentrated on $u=k$ and $u=-k$ with the property that $\mu_{k}^{+}\to\mu_{s}^{+},\mu_{k}^{-}\to\lambda_{s}^{-}$
in the narrow topology of measures (that is, with test funciton $\varphi\in L^{\infty}(\Omega)\cap C(\Omega)$)
and such that 
\[
\int_{\{|u|<k\}}|\nabla u|^{p-2}\nabla u\nabla\phi dx=\int_{\{|u|<k\}}\phi d\mu_{0}+\int_{\Omega}\phi d\lambda_{k}^{+}-\int_{\Omega}\phi d\lambda_{k}^{-},
\]
for every $\phi\in W_{0}^{1,p}(\Omega)\cap L^{\infty}(\Omega)$. 

\end{Def}

Thanks to \cite{Dal-Murat-Orsina-Prignet}, we have the important
stability result. 

\begin{thm}\label{Weak convergence p-Laplace}

Let $1<p\le N$, $\mu=\mu_{0}+\mu_{s}^{+}-\mu_{s}^{-}$ with $\mu_{0}=F-{\rm div}g\in\mathfrak{M}_{0}(\Omega)$
and $\mu_{s}^{+},\mu_{s}^{-}\in\mathfrak{M}_{s}(\Omega)\cap\mathfrak{M}^{+}(\Omega).$
Let $\mu_{n}=F_{n}-{\rm div}g_{n}+\rho_{n}-\eta_{n}$ with $F_{n}\in L^{1}(\Omega)$,
$g_{n}\in L^{p'}(\Omega)$ and $\rho_{n},\eta_{n}\in\mathfrak{M}_{b}^{+}\cap\mathfrak{M}_{s}$.
Assume $F_{n}\overset{L^{1}}{\rightharpoonup}F$ weakly, $g_{n}\overset{L^{p'}}{\to}g$
strongly and ${\rm div}g_{n}$ is bounded in $\mathfrak{M}_{b}$;
Assume also that $\rho_{n}\to\mu_{s}^{+},\eta_{n}\to\mu_{s}^{-}$
in the narrow topology. If $u_{n}$ is a sequence renormalized solutions
to (\ref{N-Laplacian equation}) with $\mu=\mu_{n}$, then, up to
a subsequence, it converges a.e.in $\Omega$ to a renormalized solution
$u$ of (\ref{p-Laplace equation}). Furthermore, $T_{k}(u_{n})$
converges to $T_{k}(u)$ in $W_{0}^{1,p}(\Omega)$ for any $k>0$. 

\end{thm}

Also we need the following estimate proved in \cite{Phuc-Verbitsky1},
which extends Theorem 1.6 of \cite{Kilp and Maly Acta} slightly.

\begin{thm}\label{p-Laplace estimate}(Theorem 2.1 of \cite{Phuc-Verbitsky1})

Let $1<p\le N$, $\Omega$ be a bounded domain of $\mathbb{R}^{N}$.
Then there exists a constant $K_{1}>0$, depending only on $p$ and
$N$ such that if $\mu\in\mathfrak{M}_{b}^{+}(\Omega)$ and $u$ is
a nonnegative renormalized solution of (\ref{p-Laplace equation})
with data $\mu,$ there holds
\[
\frac{1}{K_{1}}W_{1,p}^{\frac{d(x,\partial\Omega)}{3}}[\mu](x)\le u(x)\le K_{1}W_{1,p}^{2{\rm diam}(\Omega)}[\mu](x),\forall x\in\Omega
\]
 where the positive constant $K_{1}$ only depends on $n,p$.

\end{thm}

The following result is also due to \cite{Phuc-Verbitsky1}. 

\begin{thm}\label{p laplace existence and comparison}(Lemma 6.9
of \cite{Phuc-Verbitsky1})

Let $\mu,\nu\in\mathfrak{M}_{b}^{+}(\Omega)$ and $\mu(D)\ge\nu(D)$
for all $D\subset\Omega$. Suppose that $u$ is a renormalized solution
of 
\[
\begin{cases}
-\Delta_{p}u & =\mu\,\,{\rm in}\,\,\Omega,\\
u & =0\,\,{\rm on}\,\,\partial\Omega.
\end{cases}
\]
Then there exists $v\ge u$ such that 

\[
\begin{cases}
-\Delta_{p}v & =\nu\,\,{\rm in}\,\,\Omega,\\
v & =0\,\,{\rm on}\,\,\partial\Omega,
\end{cases}
\]
in the renormalized sense.

\end{thm}

Now we can prove our main theorem. 

\begin{proof}(of Theorem \ref{Main Thm p-Laplace}) We are going
to define a sequence of nonnegative renormalized solutions $\{u_{m}\}_{m\in\mathbb{N}}$
by 
\[
\begin{cases}
-\Delta_{N}u_{0} & =\mu,\,\,{\rm in}\,\,\Omega,\\
u_{0} & =0,\,\,{\rm on}\,\,\partial\Omega,
\end{cases}
\]
and 
\[
\begin{cases}
-\Delta_{N}u_{m+1} & =H_{l}(u_{m})+\mu,\,\,{\rm in}\,\,\Omega,\\
u_{m+1} & =0,\,\,{\rm on}\,\,\partial\Omega.
\end{cases}
\]

Choose $M$ as in Theorem \ref{Theorem u_m upper bound}, which depends
on $N,l,K=K_{1},\delta_{0},C_{1}$, where $K_{1}$ is the constant
which appears in Theorem \ref{p-Laplace estimate}. We assume $\mu$
is a nonnegative Radon measure with $\mu(\Omega)\le M$. 

By Theorem \ref{p-Laplace estimate}, as long as we can solve the
above equations and get $u_{m},\,\,m\ge0$, we will have estimates
(\ref{u0 bound})(\ref{u_=00007Bm+1=00007D bounded by u_m}) in Theorem
\ref{Theorem u_m upper bound}, with $K$ replaced by $K_{1}$. 

Then from (\ref{Hl(W) estimate})(\ref{u_m bound}), we know $H_{l}(u_{m})\in L^{1}$.
Then from Theorem \ref{p laplace existence and comparison}, we can
continue to solve the above equation and get $u_{m+1}$ and $u_{m+1}\ge u_{m}$. 

Since 
\[
u_{m}\le2KW_{1,N}^{R}[\bar{\mu}],m\in\mathbb{N},
\]
 $u_{m}$ converges, a.e. in $\Omega$, to some function $u$ with
\[
u\le2KW_{1,N}^{R}[\bar{\mu}],\forall x\in\Omega.
\]
From (\ref{Hl(W) estimate}) we know 
\[
H_{l}(u_{m})\to H_{l}(u),\,\,{\rm in}\,\,L^{1}(\Omega).
\]
Then by Theorem \ref{Weak convergence p-Laplace}, we know $u$ is
a renormalized solution of (\ref{Weak convergence p-Laplace}\ref{N-Laplace u equation}). 

\end{proof}

\section{$N/2$-Hessian equation with Dirichlet boundary condition}

Assume $\Omega\subset\mathbb{R}^{N}$ is a bounded domain with $C^{2}$
boundary. For $k=1,\cdots,N$ and $u\in C^{2}(\Omega)$ we define
\[
F_{k}[u]=S_{k}(\lambda(D^{2}u))
\]
where $\lambda(D^{2}u)=\lambda=(\lambda_{1},\lambda_{2},\cdots,\lambda_{N})$
deontes the eigenvalues of the Hessian matrix of $D^{2}u$ and $S_{k}$
is the $k$th elementary symmetric polynomial that is
\[
S_{k}(\lambda)=\sum_{1\le i_{1}<\cdots<i_{k}\le N}\lambda_{i_{1}}\cdots\lambda_{i_{k}}.
\]
 We assume that $\partial\Omega$ is uniformly $(k-1)$-convex. 

\begin{Def}\label{k-subharmonic}

An upper-semicontinuous function $u:\Omega\to[-\infty,\infty)$ is
$k$-convex (or $k$-subharmonic) if, for any $D\Subset\Omega$ for
every function $v\in C^{2}(D)\cap C(\bar{D})$ satisfying $F_{k}[v]\le0$
in $\bar{D}$, the following implication holds
\[
u\le v\,\,{\rm on}\,\,\partial D\Longrightarrow u\le v\,\,{\rm in}\,\,D.
\]
We denote by $\Phi^{k}(\Omega)$ the set of all $k$-subharmonic functions
in $\Omega$ which are not identically equal to $-\infty$. 

\end{Def}

The following weak convergence result for $k$-Hessian operators proved
in \cite{Trudinger-Wang2} is of fundamental importance in our study.

\begin{thm}\label{k-hessian weak convergence }

Let $\Omega$ be a bounded $(k-1)$-convex domain of $\mathbb{R}^{N}$.
For each $u\in\Phi^{k}(\Omega)$, there exists a nonnegative Radon
measure $\mu_{k}[u]$ in $\Omega$ such that 
\begin{enumerate}
\item $\mu_{k}[u]=F_{k}[u]$ for $u\in C^{2}(\Omega)$;
\item If $\{u_{n}\}$ is a sequence of $k$-convex functions which converges
a.e. to $u$, then $\mu_{k}[u_{n}]\to\mu_{k}[u]$ in the weak sense
of measures. 
\end{enumerate}
\end{thm}

Also we need the following results proved in \cite{Phuc-Verbitsky1}.

\begin{thm}\label{k Hessian Bdd by wolff potential}

Let $\mu$ be a nonnegative finite Radon measure such that 
\[
\mu=\mu'+f
\]
where $\mu'$ is a nonnegative Radon measure compactly supported in
$\Omega,$ and $f\ge0,f\in L^{s}(\Omega)$ with $s>\frac{N}{2k}$,
if $1\le k\le\frac{N}{2}$, and $s=1$ if $\frac{N}{2}<k\le N$. Suppose
that $-u$ is a nonpositive $k$-subharmonic function in $\Omega$,
that is continuous near $\partial\Omega$ and solves the equation
\[
\begin{cases}
F_{k}[-u] & =\mu\,\,{\rm in}\,\,\Omega,\\
u & =0\,\,{\rm on}\,\,\partial\Omega.
\end{cases}
\]
Then there is a constant $K=K(N,k)>0$ such that
\[
u(x)\le KW_{\frac{2k}{k+1},k+1}^{2{\rm diam}(\Omega)}\mu
\]
for every $x\in\Omega$. 

\end{thm}

\begin{thm}\label{k-hessian comparison}

Let $\Omega,\mu,\phi$ and $u$ be as in the above theorem. Suppose
$\nu$ is another measure with $\nu=\nu'+g$, where $\nu'$ is a nonneative
measure compactly supported in $\Omega$, and $g\ge0,g\in L^{s}(\Omega)$
with $s>\frac{N}{2k},$ if $1\le k\le\frac{N}{2}$ and $s=1$ if $\frac{N}{2}<k\le N$.
Then there exists a function $\omega$ such that $-\omega\in\Phi^{k}(\Omega)$,
$\omega\ge u$ and 
\[
\begin{cases}
F_{k}[-\omega] & =\mu+\nu,\,\,{\rm in}\,\,\Omega\\
\omega & =0,\,\,{\rm on}\,\,\partial\Omega.
\end{cases}
\]

\end{thm}

In the end we prove Theorem \ref{Main theorem k hessian}. In fact
the proof is similar to the proof of Theorem \ref{Main Thm p-Laplace}
except for that we replace Theorem \ref{Weak convergence p-Laplace},
\ref{p-Laplace estimate}, \ref{p laplace existence and comparison}
by Theorem \ref{k-hessian weak convergence }, \ref{k Hessian Bdd by wolff potential},
\ref{k-hessian comparison}. One difference is that we use $W_{\frac{2N}{N+2},\frac{N+2}{2}}^{D}[\mu]$
instead of $W_{1,N}^{D}[\mu]$. Another difference is that when we
apply Theorem \ref{k Hessian Bdd by wolff potential}, \ref{k-hessian comparison},
the measure $\mu$ needs to be sufficiently regular near the boundary.
Notice that from Theorem \ref{Theorem u_m upper bound}, we have $H_{l}(u_{m})\in L^{s}(\Omega)$
in the iteration procedure.

\end{document}